\documentclass{article}
\usepackage{authblk}
\usepackage{xspace}
\usepackage[a4paper,outer=3cm]{geometry}

\usepackage[utf8]{inputenc}
\usepackage[T1]{fontenc}
\usepackage{times}

\usepackage{xspace}

\usepackage{longtable}
\usepackage{amsmath,amsthm,amssymb,amsthm, stmaryrd}
\usepackage{enumerate}

\usepackage{algorithm}
\usepackage{algpseudocode}
\algnewcommand\algorithmicinput{\textbf{Input:}}
\algnewcommand\algorithmicoutput{\textbf{Output:}}
\algnewcommand\Input{\item[\algorithmicinput]}
\algnewcommand\Output{\item[\algorithmicoutput]}

\usepackage{hyperref}

\usepackage{subcaption}
\usepackage{pgf,tikz}
\usepackage{svg}
\usetikzlibrary{arrows}
\usetikzlibrary{decorations.markings}

\usepackage{pdftricks}
\begin{psinputs}
	\usepackage{pstricks,pst-plot,pst-node,pst-func}
\end{psinputs}
\usepackage{graphics,graphicx}

\tikzstyle{vertex}=[circle, draw, inner sep=0pt, minimum size=6pt]

\tikzset{->-/.style={decoration={
  markings,
  mark=at position .5 with {\arrow{>}}},postaction={decorate}}}

 \usepackage{marginnote}

\newcommand{\m}[1]{}

\usepackage[nothm]{thmbox}
\usepackage{amsthm}
\usepackage{thmtools}

\declaretheorem[parent=section,thmbox=M]{theorem}

\declaretheorem[numberlike=theorem,thmbox=M]{proposition}

\declaretheorem[numberlike=theorem,thmbox=M]{conjecture}

\declaretheorem[numberlike=theorem]{lemma}

\declaretheorem[numberlike=theorem]{claim}

\newenvironment{subproof}{\par\noindent {\it Proof}.\ }{\hfill$\blacklozenge$ \par\vspace{11pt}}

\newcommand{\Gya}{Gy\'arf\'as\xspace}

\newcommand{\Chu}{Chudnovsky\xspace}

\newcommand{\sm}{\setminus}
\newcommand{\mc}{\mathcal}

\newcommand{\mC}{\mathcal{C}}

\usepackage{esvect}
\newcommand{\ra}{\rightarrow}

\newcommand{\ora}[1]{\vv
{#1}}

\newcommand{\ob}[1]{\overbar{#1}}

\DeclareMathOperator{\dic}{\ora \chi}

\newcommand{\dP}[1]{\ora{P}_{\hspace{-0.1em}#1}}
\newcommand{\dC}[1]{\ora{C}_{\hspace{-0.1em}#1}}
\newcommand{\dK}[1]{\ora{K}_{\hspace{-0.1em}#1}}
\newcommand{\TT}[1]{\ora{TT}_{\hspace{-0.1em}#1}}

\newcommand{\F}{Forb_{ind}}

\newcommand{\overbar}[1]{\mkern 1.7mu\overline{\mkern-1.7mu#1\mkern-1.7mu}\mkern 1.7mu}

\tikzstyle{vertex}=[circle,draw, top color=gray!5, 
	    bottom color=gray!30, minimum size=12pt, scale=1, inner sep=0.5pt]
\tikzstyle{arc}=[->, > = latex',  thick]
\tikzstyle{edge}=[thick, blue]

\title{$(\dP6, \text{ triangle})$-free digraphs have bounded dichromatic number}

\author[1]{Pierre Aboulker}
\author[1,2]{Guillaume Aubian}
\author[2]{Pierre Charbit}
\author[3]{St\'ephan Thomass\'e}

\affil[1]{DIENS, \'Ecole normale sup\'erieure, CNRS, PSL University, Paris, France.}
\affil[2]{Université de Paris, CNRS, IRIF, F-75006, Paris, France.}
\affil[3]{Laboratoire d'Informatique du Parallélisme, \'Ecole Normale Sup\'erieure de Lyon, 69364 Lyon, Cedex 07, France}

\begin{document}

\maketitle

\begin{abstract}
The dichromatic number of an oriented graph is the minimum size of a partition of its vertices into acyclic induced subdigraphs. We prove that oriented graphs with no induced directed path on six vertices and no triangle have bounded dichromatic number. This is one (small) step towards the general conjecture asserting that for every oriented tree $
T$ and every integer $k$, any oriented graph that does not contain an induced copy of $T$ nor a clique of size $k$ has dichromatic number at most some function of $k$ and $T$.

\end{abstract}

\section{Introduction}

In this paper, we only consider \emph{graphs} or \emph{directed graphs} (\emph{digraphs} in short) with no loops, no parallel edges or arcs nor anti-parallel arcs (in particular our digraphs contains no cycle of length $2$).

Given an undirected graph $G$, we denote by $\omega(G)$ the size of a maximum clique of $G$ and by $\chi(G)$ its chromatic number.  
A class of graphs $\mC$ is \emph{$\chi$-bounded} if there exists a function $f$, such that every graph $G$ in $\mC$ satisfies $\chi(G)\leq f(\omega(G))$. 

Given a graph (resp. a digraph) $H$, we denote by $\F(H)$ the class of graphs (resp. digraphs) that do not contain $H$ as an induced subgraph (resp. induced subdigraph). 
A celebrated and still wide open question in the area of graph colouring is the following conjecture of \Gya~\cite{G75} and Sumner~\cite{S81} (see~\cite{SS20} for a survey on $\chi$-boundedness). 

\begin{conjecture}[\Gya-Sumner]
For any forest $F$, $\F(F)$ is $\chi$-bounded.
\end{conjecture}

In this paper, we study an analogue of this conjecture for digraphs. For a digraph $D$ we denote by $\omega(D)$ the clique number of the underlying graph of $D$ and by $\dic(D)$ its {\em dichromatic number}, that is the minimum integer $k$ such that the set of vertices of $D$ can be partitioned into $k$ acyclic subdigraphs. A class of digraphs $\mC$ is $\dic$-bounded if there exists a function $f$ such that every digraph $D$ in $\mC$ satisfies $\dic(D)\leq f(\omega(D))$. 

\begin{conjecture}[Aboulker, Charbit, Naserasr \cite{ACN21}]\label{conj:tree_chi_bounded}
For any oriented forest $\ora F$, $\F(\ora F)$ is $\dic$-bounded.
\end{conjecture}

It is enough to prove it for oriented trees (the proof is the same as for the undirected case, and can be found in~\cite{S21}, Proposition 1.6). 
An \emph{oriented star} is an oriented tree with at most one non-leaf vertex. 
\Chu, Scott and Seymour~\cite{CS19} proved it for oriented stars  as well as for two of the four possible orientations of the path on $4$ vertices: $\rightarrow \leftarrow \leftarrow$ and $\leftarrow \rightarrow \rightarrow$ (they actually prove that  for any integer $k$ and any oriented graph $\vec H$ where $\vec H$ is either an oriented star or $\rightarrow \leftarrow \leftarrow$, or $\leftarrow \rightarrow \rightarrow$, any digraph in $\F(\vec H)$ with clique number at most $k$ can be partitioned into a bounded number of stable sets, which is clearly stronger). Cook,  Masar\'ik,  Pilipczuk, Reinald and Souza~\cite{CMPRS22} proved it for the two other  orientations of the paths on $4$ vertices: $\rightarrow\rightarrow\rightarrow$ and $\rightarrow \leftarrow \rightarrow$. Nothing more is known.

Proving the conjecture for directed paths is already a very challenging case. 
In this paper, we go a step further in this direction by proving the following, where $\dP6$ denotes the directed path on $6$ vertices.

\begin{theorem}\label{thm:main_thm}
For every $D \in \F(\dP6)$ with $\omega(D) \leq 2$, $ \dic(D) \leq 382$.
\end{theorem}

Note that we did not try to optimise the bound.

\paragraph{Context and Related Works}

 It has been a central question in graph theory over the past 40 years to understand what substructures are forced by large chromatic number. Or equivalently, which are the substructures that, if forbidden, result in bounded chromatic number. The notion of $\chi$-boundedness deals with this question.

Similarly, the notion of $\dic$-boundedness deals with the analogue question for digraphs and dichromatic number: a class $\cal C$ is $\dic$-bounded if for every $k$ there exists a value $\phi_k$ such that any digraph in $\mC$ with dichromatic number larger than $\phi_k$ must contain some orientation of a clique (or {\em tournament}) on $k$ vertices. It turns out that the acyclic tournament on $k$ vertices (denoted by $\TT k$) is sufficient to characterize this notion: indeed every tournament on $2^k$ vertices contains $\TT k$, and therefore a class $\cal C$ is $\dic$-bounded if and only if for every $k$ there exists a value $c_k$ such that any digraph in $\mC$ of dichromatic number at least $c_k$ contains a $\TT k$.

More generally, given a class of digraphs $\mC$, a digraph $H$ is a \emph{hero in $\mC$} if there is a constant $c_H$ such that digraphs of $\mC$ that do not contain $H$ as an induced subdigraph have dichromatic number at most $c_H$. The discussion above is that a class $\mC$ is $\dic$-bounded if and only if for every integer $k$, $\TT k$ is a hero in $\mC$, and Conjecture~\ref{conj:tree_chi_bounded} can be rephrased as : \emph{for every oriented forest $\ora F$, for every integer $k$, $\TT k$ is a hero in $\F(\ora F)$}.  Additionally, a result of~\cite{HM12} implies that if $H$ is {\em not} an oriented forest, then no digraph is a hero in $\F(H)$ except for $K_1$ (the digraph on one vertex) and $\TT 2$. 
 
  In a seminal paper, Berger et al.~\cite{hero} give a simple inductive characterization of  heroes in the class of tournaments (these contain, of course, $\TT k$, but many more).
 Note that if a class of digraphs $\mathcal C$ contains all tournaments, then a hero in $\mc C$ is in particular a hero in tournaments, but a hero in tournaments does not need to be a hero in $\mathcal C$. Every class considered in the following contains all tournaments.

 Let $\ob{K_k}$ be the digraph on $k$ vertices with no arc and observe that the class of tournaments is the same as $\F(\ob{K_2})$. 
 Harutyunyan et al.~\cite{HLNT19} extended the above result of Berger at al. by proving that, for every $k \geq 3$, heroes in $\F(\ob{K_k})$ are the same as heroes in tournaments.

Following these works, a systematic study of heroes in classes of digraphs of the form $\F(\ora F)$ where $\ora F$ is an oriented forest has been initiated in~\cite{ACN21}. 
In particular, it is proved that if $\ora F$ is not a disjoint union of oriented stars, then the only possible heroes in $\F(\ora F)$ are the transitive tournaments.  
On the other hand, it was conjectured in~\cite{ACN21} that if $\ora S$ is a disjoint union of oriented stars, then heroes in $\F(\ora S)$ are the same as heroes in tournaments, but this turned out to be false. 
In the paragraph below, we give a quick overview of the results on this particular question.

  The result of \Chu et al.~\cite{CS19} mentioned earlier implies that transitive tournaments are heroes in $\F(\ora S)$ for any disjoint union of oriented stars $\ora S$. 
In~\cite{AACmulti}, it is proved   that heroes in $\F(\dP3)$ are  the same as heroes in tournaments. 
Denote by $K_1 + \TT 2$ the disjoint union of $K_1$ and $\TT 2$ and observe that $\F(K_1 + \TT 2)$ is the class of oriented complete multipartite graphs. Heroes in $\F(K_1 + \TT 2)$ have  been investigated in~\cite{AACmulti} where it is proved that they form a strict super class of transitive tournaments, and a strict subclass of heroes in tournaments (which disproved the aforementionned conjecture of \cite{ACN21}). 
Finally,  heroes in $\F(\dK{1,2})$ (where $\dK{1,2}$ denotes the oriented star on three vertices with a vertex of out-degree $2$, digraphs in this class are called  \emph{locally-out tournaments}) were studied in ~\cite{AAC21} and~\cite{S21} (they are still conjectured to be the same as heroes in tournaments).
\medskip  

A digraph is \emph{$t$-chordal} if all its induced directed cycle have length $t$. Surprisingly, for every $t \geq 3$, the class of $t$-chordal digraphs has been proved~\cite{CHMS22} to not be $\dic$-bounded. 
Note that $t$-chordal digraphs are defined by forbidding an infinite number of digraphs, contrary to results mentioned above. 

An {\em oriented chordal graph} is an orientation of a chordal graph. This is again a class of digraphs with a distinct flavour, obtained by taking all possible orientations of a class of (undirected) graphs.  Heroes in oriented chordal  graphs have been fully characterised in~\cite{AAS22}.


\section{Definitions}

If $D$ be a digraph. We denote by $V(D)$ its set of vertices and by $A(D)$ its set of arcs. For $X\subset V(D)$ we define $N^+(X)=\{y\in V(D)\setminus X, \exists x\in X \text{ such that } xy\in A(D)\}$ and $
N^-(X)=\{y\in V(D)\setminus X, \exists x\in X \text{ such that } yx\in A(D)\}$. A {\em subdigraph} of $D$ is a digraph obtained from $D$ by removing some arcs and some vertices (with all arcs incident to these vertices). If only vertices are removed, it is  an {\em induced subdigraph}. For a given set of vertices $X \subseteq V(D)$, we denote by $D[X]$ the induced subdigraph obtained by removing $V(D)\setminus X$. Given a set of digraphs $\mc H$, we say that a digraph $D$ is \emph{$\mc H$-free} if it contains no induced subdigraph isomorphic to some member of $\mc H$. We denote by  $\F(\mc H)$ the class of $\mc H$-free digraphs. 
We say that $D$ is  {\em triangle-free} if $\omega(D) \leq 2$. 
Given a digraph $H$ we say that $D$ does not \emph{contain} (or \emph{has no}) $H$ if $D$ does not contain $H$ as a (not necessarily induced) subdigraph. 

We write $x \ra y$ when $xy \in A(D)$. 
A {\em trail} of a digraph $D$ is a sequence of vertices $x_1x_2\ldots x_p$ such that $x_ix_{i+1}\in A(D)$ for each $i<p$ and each arc is used once (but vertices can be used several times). 
It is \emph{closed} if $x_1 = x_p$ and its {\em length} is its number of arcs. We say {\em odd closed trail} for a closed trail of odd length. 
A  trail (resp. closed  trail) in which  vertices  are pairwise distinct is called a {\em directed path} (resp. {\em directed cycle}). 
The directed path of length $k-1$ is denoted by $\dP{k}$.

 A \emph{$k$-dicolouring} of $D$ is a partition of $V(D)$ into $k$ sets $V_{1}, \dots, V_{k}$ such that $D[V_{i}]$ is acyclic for every $i = 1, \dots, k$. 
 The \emph{dichromatic number} of $D$, denoted by $\dic(D)$ and introduced by  Neumann-Lara~\cite{NL82}  is the minimum integer $k$ such that $D$ admits a $k$-dicolouring. 
We will sometimes extend $\dic$ to subsets of vertices, using $\dic(X)$ to mean $\dic(D[X])$ where $X \subseteq V(D)$. For a set $\mC$ of digraphs  we write $\dic(\mC)$ to denote the maximum of $\dic(D)$ over all elements $D$ in $\mC$, and write $\dic(\mC)=\infty$ if this is not bounded. 


\section{Preliminaries}

A set of vertices $X$ is \emph{dipolar} if for every $x \in X$, $N^+(x) \subseteq X$ or $N^-(x) \subseteq X$. 
This notion was first introduced in \cite{ACN21} under the name "nice set" and has been renamed "dipolar set" in \cite{CMPRS22}. 
The main tool using dipolar sets is the following lemma. We include its proof because it is short and enlightening for people unfamiliar with the dichromatic number. 

\begin{lemma}[Lemma 17 in~\cite{ACN21}]\label{lem:dipolar}
Let $\mc C$ be a class of digraphs closed under taking induced subdigraph. Suppose that there exists a constant $c$ such that each digraph $D \in \mathcal C$ has a dipolar set $S$ such that $\dic(S) \leq c$. Then $\dic(\mc C) \leq 2c$. 
\end{lemma}

\begin{proof}
Let $D \in \mc C$ be a minimal counter example, that is: $\dic(D)=2c+1$ and for every proper  subdigraph $H$ of $D$, $\dic(H)\le 2c$. 
By the hypothesis, $D$ admits a dipolar set $S$,  such that $\dic(S) \leq c$. 
Set $S^+ = \{x \in S \mid N^-(x) \subseteq S\}$ and $S^- = \{x \in S \mid N^+(x) \subseteq S\}$. By definition of a dipolar set, $S= S^+ \cup S^-$. 

The key observation is that any directed cycle that intersects $S$ and $V(D) \sm S$  intersects  both $S^+$ and $S^-$. 
Hence, by minimality of $D$, we can dicolour  $V(D) \sm S$ with $2c$ colors. We can then extend this dicoloring to $D$ by using colours $1, \dots, c$ for $S^+$ and $c+1, \dots, 2c$ for $S^- \setminus S^+$. 
\end{proof}

The strategy to prove our result is to show that every digraph in our class has a dipolar set with dichromatic number at most $191$ and then apply Lemma~\ref{lem:dipolar}. 
The next two results give simple techniques to bound the dichromatic number of a digraph, they  will  be extensively used to prove that the dichromatic number of some dipolar set is bounded. The first one is probably well known but we don't have any reference for it, the proof is very short.

\begin{lemma}\label{lem:odd_cycle}
If a digraph $D$ does not contain odd directed cycles as subdigraphs, then $\dic(D) \leq 2$. 
\end{lemma}

\begin{proof}
Let $D $ be a digraph with no odd directed cycle and since the dichromatic number of a digraph is the maximum of the dichromatic number of its strong components, we can assume without loss of generality that $D$ is strongly connected. In that case, we prove that the underlying graph $G$ of $D$ is in fact bipartite. Assume by contradiction $G$ contains an odd cycle $C = c_1\ra c_2\ra  \dots\ra  c_{2k+1}\ra c_1$. For $i = 1, \dots, {2k + 1}$, let $P_i$ be a shortest directed path from $c_i$ to $c_{i + 1}$ (indices being taken modulo $2k + 1$). Observe that either $P_i = c_ic_{i+1}$, or $c_{i+1}c_i \in A(D)$, in which case $P_i$ has odd length, for otherwise $P_i \cup \{c_{i+1}c_{i}\}$ is an odd directed cycle.  Hence the union of the $P_i$ for $i = 1 \dots 2k + 1$ forms a closed odd trail, which contains an odd directed cycle, a contradiction. 
\end{proof}

The next result is the dichromatic version of the celebrated Gallai-Roy-Vitaver theorem asserting that the chromatic number is upper-bounded by the largest size of a directed path. In a nutshell: the dichromatic number is upper-bounded by the largest size of a directed path of some feedback arc set.

\begin{proposition}\label{prop:noPk}
Let $D$ be a digraph. Given a total ordering of the vertices of $D$, we say that an arc $xy$ is forward if $x$ precedes $y$ in this ordering, and backward otherwise. The two following propositions are equivalent 
\begin{itemize}
    \item $\dic(D)\leq k$
    \item There exists an ordering of the vertices of $D$ such that there exists no directed path on $k+1$ vertices consisting only of backward arcs.
\end{itemize}
\end{proposition}

\begin{proof}

One direction is easy : if $\dic(D)\leq k$ then there exists a partition $(C_1,C_2,\ldots C_k)$ of $V(D)$ with $C_i$ inducing an acyclic digraph. We construct an order on $V(D)$ by putting all vertices of $C_i$ before all vertices of $C_{i+1}$ for each $i$ and within each class we use a topological sort. It is clear that in the resulting order, there can be no patch on more than $k$ vertices where all  arcs go backward since a backward arcs goes from one class to a previous one.

For the converse direction, assume that $D$ has an ordering on its vertices such that there exists no directed path on $k+1$ vertices consisting only of backward arcs and let us prove that $D$ is $k$-dicolourable. 
For every $x \in V(D)$, define $f(x)$ the maximum number of vertices in a path consisting only of backward arcs and ending in $x$. By definition $1\leq f(x)\leq k$. Define $C_i=f^{-1}(i)$ and let us prove that $C_i$ does not contain any backward arc. Assume by contradiction $xy$ is such an arc. Then there exists a path on $i$ vertices ending in $x$ consisting only of backward arcs, which implies that $f(y)\geq i+1$, contradiction. So each $C_i$ induces an acyclic digraph, and thus $\dic(D)\leq k$.
\end{proof}

The last lemma of this section is used to find induced directed paths.

\begin{lemma}\label{lem:two_push}
Let $D$ be a triangle-free digraph, $C$ a (not necessarily induced) odd directed cycle of $D$ and $a \in N(C)$. Then there exists consecutive vertices $b\rightarrow c\rightarrow d$ of $C$ such that
\begin{itemize}
    \item either $a\ra b\ra c\ra d$ is an induced $\dP4$,
    \item or $b\ra c\ra d\ra a$  is an induced $\dP4$,
    \item or $a\ra b\ra c\ra d$ is a $C_4$ (in particular, $a\in N^{+}(C)\cap N^{-}(C)).$
\end{itemize}
\end{lemma}

\begin{proof} Assume $a\in N^-(C)$. Let us denote by $x_1,\ldots,x_{2k+1}$ the vertices of $C$ (i.e. $\forall i\leq 2k$, $x_ix_{i+1}\in A(D)$ and $x_{2k+1}x_1\in A(D)$). Assume without loss of generality that $ax_{1}\in A(D)$. 
Let $1 \leq p \leq k$ be the maximum integer such that $ax_{2p+1} \in A(D)$.
Since the digraph is triangle-free, $ax_{2k+1}\notin A(D)$, so $p \leq k$. 
 It is straightforward to see that $b=x_{2p+1}$, $c=x_{2p+2}$, $d=x_{2p+3}$  satisfies either the first or third item of the lemma. By reversing the arcs of the digraph, the same proof works if $a\in N^+(C)$. 
\end{proof}

We will often use this lemma the following way : if $a\in N^+(C)\setminus N^-(C)$ (resp. $a\in N^-(C)\setminus N^+(C)$), then the first (resp. the second) output holds.


\section{Proof of Theorem~\ref{thm:main_thm}}
For a subset $X$ of vertices, we define recursively the sets $N_k^+(X)$, $N_k^{-}(X)$ and $N_k(X)$ by $N_{0}^+(X)=N_{0}^-(X)=N_0(X)=X$, and for $k\geq 1$ :
\begin{alignat*}{2}
&N_{k}^+(X)&&=N^+(N_{k-1}^+(X))\setminus \bigcup_{i<k} N_i(X) \\
&N_{k}^-(X)&&=N^-(N_{k-1}^-(X))\setminus \bigcup_{i<k} N_i(X) \\
&N_{k}(X)&&=N_{k}^+(X)\cup N_k^-(X)
\end{alignat*}
We gather in the following claim several straightforward facts that we will use in the proof.
\begin{claim}\label{clm:basicsets} For any $X\subset V$, the following hold
\begin{enumerate}
    \item\label{it:1} $N^{+}_1(X) = N^+(X)$, $N^{-}_1(X) = N^-(X)$ and $N_1(X) = N^+(X)\cup N^-(X)$
    \item\label{it:2} There are no arcs between $X$ and  $N_k(X)$ for $k>1$.
    \item\label{it:3} If $x\in N_{k-1}(X)$, then either $N^+(x)\subseteq \bigcup_{i\leq k} N_i(X)$ or $N^-(x)\subseteq \bigcup_{i\leq k} N_i(X)$.  
 
    \item\label{it:4} If $x\in N_{k}^+(X)$ (resp $ N_{k}^-(X)$), there exists a directed path $x_0x_1\ldots x_k$ (resp. $x_kx_{k-1}\ldots x_0$)  such that $x_k=x$ and $x_i\in N_{i}^+(X)$ for every $i\geq 0$.
    
\end{enumerate}
\end{claim}
Items 1), 2) and 3) follow from the definition and 4) is easy to prove by induction on $k$.\\

Let now $D$ be a triangle-free digraph in $\F(\dP6)$. 
Let $C=x_1x_2\dots x_{2k+1}x_1$ be a  (not necessarily induced) odd directed cycle of $D$ of minimum length (we may assume it exists by Lemma~\ref{lem:odd_cycle}). 
During the proof, for simplicity, we write $C$ for $V(C)$, $D[C]$ for $D[V(C)]$ and $N_{k}(C)$ for $N_{k}(V(C))$.

We are going to prove that the set 
$$S= C \cup N(C) \cup N_{2}(C) \cup N_{3}(C)$$ 
is dipolar and has dichromatic number at most $191$, which implies Theorem~\ref{thm:main_thm} by Lemma~\ref{lem:dipolar}. 

\begin{claim}\label{clm:Sdipo}
$S$ is dipolar. Moreover, $\dic(N_{3}(C)) \leq 2$. 
\end{claim}

\begin{subproof}
To prove that $S$ is dipolar, we need to prove that for every vertex $x$ in $S$, etiher $N^+(x)$ or $N^-(x)$ is contained in $S$. Note that by Claim \ref{clm:basicsets} item 3, this is trivial if $x\in C\cup N_1(C) \cup N_2(C)$. 

Assume now that  $x \in N^{+}_3(C)$ and let us prove that $N^+(x) \subseteq N(C) \cup N_{2}(C)$, which will imply both parts of the claim, since this proves that $N^+_3(C)$ is an independent set.  

By Claim \ref{clm:basicsets} item 4, there exists a directed path $x_0\ra x_1\ra x_2\ra x_3$, where $x_3=x$ and $x_i\in N^+_i(C)$. If $x_1 \in N^{+}(C) \setminus N^{-}(C)$, then, by Lemma~\ref{lem:two_push}, there exists $a,b,c \in C$ such that $abcx_1$ is an induced $\dP4$. Since there is no arc between $C$ and $N_{2}(C) \cup N_{3}(C)$ (by Claim \ref{clm:basicsets} item \ref{it:2}) and $D$ is triangle-free, $a\ra b\ra c\ra x_1\ra x_2\ra x_3$ is an induced $\dP6$, a contradiction. 

So we can assume $x_1 \in N^{+}(C) \cap N^{-}(C)$. Consider $y\in N^+(x)$, and let us prove that $y\in N(C) \cup N_2(C)$. Let $t$ be an in-neighbour of $x_0$ in $C$ and observe that $t\ra x_0\ra x_1\ra x_2\ra x_3\ra y$ is a $\dP6$ and the only way for it not to be induced (because of (Claim \ref{clm:basicsets} item \ref{it:2})) is that $y$ is adjacent with one of $\{t,x_0,x_1\}$. If $y$ is adjacent with $t$ or $x_0$, then $y \in N(C)$. If $y$ is adjacent with $x_1$, and since $x_1 \in N^{+}(C) \cap N^{-}(C)$, we get that $y \in  N_{2}(C)$. We thus have proven that $y\in N(C) \cup N_2(C)$. Similarly, if $x \in N^{-}_3(C)$, then $N^-(x) \subseteq N(C) \cup N_{2}(C)$, which concludes the proof of this claim.

\end{subproof}

\begin{claim}\label{clm:C_bounded}
$\dic(D[C]) \leq 3$. 
\end{claim}

\begin{subproof}
By minimality of $C$, removing a vertex from $C$ yields a digraph with no odd directed cycle, which thus has dichromatic number at most $2$ by Lemma~\ref{lem:odd_cycle}.
\end{subproof}

\begin{claim} \label{clm:N+moinsN-_bounded}
$\dic(N^{+}(C) \setminus N^{-}(C)) \leq 4$ and  $\dic(N^{-}(C) \setminus N^{+}(C)) \leq 4$. 
\end{claim}

\begin{subproof}
Let us prove that $\dic(N^{+}(C) \setminus N^{-}(C)) \leq 4$.  
We first prove that $N^+(x_1) \cup N^+(x_2)$ intersects all odd directed  cycles of $N^{+}(C) \setminus N^{-}(C)$. Suppose that it is not the case, and let $C'$ be such an odd directed  cycle. 
Let $i\geq 3$ be minimum such that $x_i$ has an out-neighbour in $C'$ (so that $x_1, \dots, x_{i-1}$ don't). Since $C'\subset N^+(C)\setminus N^-(C)$, $x_i$ does not have an in-neighbour in $C'$, so by Lemma~\ref{lem:two_push} applied to $C'$, there are $3$ consecutive vertices $a, b, c$ of $C'$, such that $x_i \ra a \ra b \ra c$ is an induced $\dP4$. By the choice of $i$, we then have that $x_{i-2} \ra x_{i-1} \ra x_i \ra a \ra b \ra c$ is an induced $\dP6$, a contradiction.
Now, $N^{+}(C) \setminus N^{-}(C)$ can be partitioned into two stable sets and a digraph with no odd directed cycle, and thus be $4$-dicoloured.  
By directional duality,  $\dic(N^{-}(C) \setminus N^{+}(C)) \leq 4$.
\end{subproof}

\begin{claim}\label{clm:N2+moinsN2-_bounded}
   $\dic(N^{+}_2(C) \setminus N^{-}_2(C)) \leq 2$ and $\dic(N^{-}_2(C) \setminus N^{+}_2(C)) \leq 2$. 
\end{claim}

\begin{subproof}
We prove that $\dic(N^{+}_2(C) \setminus N^{-}_2(C)) \leq 2$. Assume by contradiction this is not the case, so that by Lemma~\ref{lem:odd_cycle} we get an odd directed cycle $C'$ in $N^{+}_2(C) \setminus N^{-}_2(C)$ .  Let $u$ be a vertex in $N^{+}(C) \cap N^-(C')$, which is non empty by definition of $N^{+}_2(C)$.

If $u \in N^{+}(C) \setminus N^{-}(C)$, then by Lemma~\ref{lem:two_push}, there exist $a,b,c \in C$ such that $a \ra b \ra c \ra u$ is an induced $\dP4$, which along with a vertex $v \in N^+(u) \cap V(C')$ and the out-neighbour of $v$ in $V(C')$ forms an induced $\dP6$, a contradiction (remember that by Claim \ref{clm:basicsets} Item \ref{it:2}, there is no arc between $C$ and $C'$). 

Thus $u \in N^{+}(C) \cap N^{-}(C)$ and since $V(C')$ is disjoint from $N^{-}_2(C)$, $u$ has no in-neighbour in $V(C')$. Hence, by Lemma~\ref{lem:two_push} applied on $C'$, there exist $a,b,c \in V(C')$ such that $u\ra a \ra b \ra c$ is an induced $\dP4$, which along with any $v \in N^-(u) \cap C$ and the in-neighbour of $v$ in $C$ forms an induced $\dP6$, a contradiction.
\end{subproof}

\begin{claim}\label{clm:hard}
$\dic(N^{+}(C) \cap N^{-}(C)) \leq 30$. Moreover, if for every $x \in C$, both $N_2^+(x)$ and $N_2^-(x)$ are  stable sets, then $\dic(N_2^+(C) \cap N_2^-(C)) \leq 30$. 
\end{claim}

\begin{subproof} 
The same proof works for the two assertions of the claim. Let $\ell \in \{1,2\}$ and observe that, by hypotheses (triangle-free for $\ell=1$, or the assumption of the second sentence for $\ell=2$), for every $x \in C$, both $N^{\ell +}(x)$ and $N^{\ell -}(x)$ are stable sets.

Let $X = (N^{\ell+}(C) \cap N^{\ell-}(C)) \setminus N^{\ell}(\{x_1, \dots, x_6\})$. It is enough to prove that $\dic(X) \leq 30-12= 18$. 


For each vertex $v  \in X$, choose (arbitrarily) a vertex $x_i$ (resp. $x_j$) in $C$ such that there is a directed path of length $l$ from $v$ to $x_i$ (resp. from $x_j$ to $v$). Set $out(v)=i$ and $in(v)=j$ so that we define two functions $out$ and $in$ from $X$ to $\{1,\ldots,2k+1\}$.

 In the case where $\ell = 2$, let  $p^{+}_{v}$ (resp. $p^{-}_{v}$) be a vertex such that $v \ra p^+_v \ra x_{out(v)}$ (resp. $x_{in(v)} \ra p^-_v \ra v$). In the rest of the proof, $v \ra p^+_v \ra x_{out(v)}$ is understood as $v \ra x_{out(v)}$ in the case where $\ell = 1$.

 For $i \in [0,5]$, let  $X_{i} = \{v \in X \mid out(v) = i\mod 6 \}$ and then define $X_{i,\geq} = \{v \in X_{i} \mid out(v) \geq in(v)\}$ and $X_{i,<} = \{v \in X_{i} \mid out(v) < in(v)\}$. It is enough to prove that $\dic(X_{i,\geq}) \leq 2$ and $\dic(X_{i,<})\leq 1$ for $i=0, \dots, 5$. 

So now $i$ is fixed and we define a total order $\prec$ on $X_i$ the  following way: we say first that $u \prec v$ when $out(u) < out(v)$ and then extend arbitrarily this partial ordering to a total ordering of $X_i$. 

We first prove that $\dic(X_{i,\geq}) \leq 2$ using Proposition~\ref{prop:noPk} applied to the reversal of $\prec$ defined above. Suppose then by contradiction that there exist $a,b,c \in X_{i,\geq}$ such that $a \prec b \prec c$ and $ab, bc \in A(D)$. Since $N^{\ell -}(x)$ is a stable set for every $x \in C$, $out(a)\neq out(b)$ and $out(b) \neq out(c)$ and thus
$$out(c) \geq 6 + out(b) \geq 12 +out(a) \geq 12 + in(a)$$
If $in(a)$ has the same parity as $out(a)$ (and thus as $out(b)$ and $out(c)$), then 
$x_{1} \ra x_2   \rightarrow \dots \rightarrow x_{in(a)} \rightarrow p^{-}_{a} \rightarrow a \rightarrow b \rightarrow c \rightarrow p^{+}_{c} \rightarrow x_{out(c)} \rightarrow \dots \rightarrow x_{2k+1} \rightarrow x_{1}$ 
is an odd  closed trail (it does need to be a directed cycle because $p^-_a = p^+_c$ is possible)  and otherwise, 
$x_{1} \ra x_2  \rightarrow \dots \rightarrow x_{in(a)} \rightarrow p^{-}_{a} \rightarrow a \rightarrow b \rightarrow p^{+}_{b} \rightarrow x_{out(b)} \rightarrow \dots \rightarrow x_{2k+1} \rightarrow x_{1}$ 
is an odd directed cycle. In both cases we get an odd directed trail that has strictly less vertices than $C$, and since an odd  closed trail contains an odd directed cycle, we get our contradiction. Thus $\dic(X_{i,\geq}) \leq 2$.

We now prove that $\dic(X_{i,<})\leq 1$. 
Suppose that there exist $a,b \in X_{i,<}$ such that $b \prec a$ and $ab \in A(D)$. 
Thus $out(b)+6 \leq out(a) < in(a) $. 
If $out(a)$ and $in(a)$ do not have the same parity, 
then $x_{out(a)} \rightarrow  x_{out(a)+1} \rightarrow \dots \rightarrow x_{in(a)} \rightarrow p^+_a \rightarrow a \rightarrow p^-_a \rightarrow x_{out(a)}$ is an odd  closed trail.  
Otherwise $out(a)$ and thus $out(b)$ have the same parity as $in(a)$, and then $x_{out(b)} \rightarrow  \dots \rightarrow x_{in(a)} \rightarrow p^-_a \rightarrow a \rightarrow b \rightarrow p^+_b \rightarrow x_{out(b)} $ is an odd directed cycle.  
In both cases it has strictly less vertices than $C$, a contradiction.  
Thus $\dic(X_{i,<}) \leq 1$ by Proposition~\ref{prop:noPk}. 
\end{subproof}

Let $\dC{3,2}$ be the digraph with vertices $u,v_1, v_2, w_1, w_2$ and arcs $uv_1, v_1v_2, v_2w_2, uw_1, w_1w_2$. Observe that if $G \in \F(\dC{3,2})$, then for every $x \in V(G)$, $N^{+}_2(x)$ and $N^{-}_2(x)$ are stable sets. Hence, by the previous claims (all of them), we get that for every triangle-free digraphs $G \in \F(\{\dP6, \dC{3,2}\})$,
the set $Q \cup N(Q) \cup N_{2}(Q) \cup N_{3}(Q)$, where $Q$ is an odd directed cycle of $G$ of minimum length, is dipolar and has dichromatic number at most $3 + 4 + 4+ 2 + 2 +1 + 1 + 30 + 30=77$. Hence, by Lemma~\ref{lem:dipolar} we get that:

\begin{claim}\label{clm:C32_bounded}
Triangle-free digraphs in $\F(\{\dP6, \dC{3,2}\})$ have dichromatic number at most $144$. 
\end{claim}

We are now able to prove the last bit of the proof. 

\begin{claim}\label{clm:N2+capN2-}
$\dic(N^{+}_2(C) \cap N^{-}_2(C)) \leq 144$.
\end{claim}

\begin{subproof}
By Claim~\ref{clm:C32_bounded}, we may assume that $N^{+}_2(C) \cap N^{-}_2(C)$ contains  $\dC{3,2}$ as an induced subdigraph. 
Thus there exists $u,v_1,v_2,w_1,w_2 \in N^{+}_2(C) \cap N^{-}_2(C)$ such that $uv_1, uw_1, v_1v_2, w_1w_2, v_2w_2 \in A(D)$. 
Moreover, there exists $r,s \in C$, and $t \in N^+(C)$ such that $rs,st,tu \in A(D)$. Now, since $r \ra s \ra t \ra u \ra v_1 \ra v_2$ is not induced, $t$ and $v_2$ are adjacent, and since $r \ra s \ra t \ra u \ra w_1 \ra w_2$ is not induced, $t$ and $w_2$ are adjacent. Hence $t,v_2,w_2$ forms a triangle, a contradiction. 
\end{subproof}

Altogether, we get that $\dic(S) \leq 3 + 4 + 4 + 30 +2 + 2 + 144 +1 +1 = 191$, and thus $\dic(D) \leq 382$.

\subsubsection*{Acknowledgement}
This research was partially supported by the ANR project DAGDigDec (JCJC)   ANR-21-CE48-0012, by the
ANR project Digraphs
ANR-19-CE48-0013, and by the group Casino/ENS Chair on Algorithmics and Machine Learning.

\bibliography{refs}

\end{document}